\documentclass[11pt]{article}
\usepackage{amsfonts}
\usepackage{mathrsfs,dsfont,bbm}

\usepackage[latin1]{inputenc}
\usepackage{amsmath,amssymb}
\usepackage{latexsym}
\usepackage{epstopdf}
\usepackage{geometry}
\usepackage[active]{srcltx}
\usepackage{graphicx}
\usepackage[
bookmarks=true,         
bookmarksnumbered=true, 
colorlinks=true, pdfstartview=FitV, linkcolor=blue, citecolor=blue,
urlcolor=blue]{hyperref}

\newtheorem{theorem}{Theorem}[section]

\newtheorem{lemma}[theorem]{Lemma}

\numberwithin{equation}{section}
\newenvironment{proof}[1][Proof]{\noindent\textit{#1.} }{\hfill \rule{0.5em}{0.5em}}
\def\AA{\mathcal{A}}

\begin{document}

\title{Higher order derivative of self-intersection local time for fractional Brownian motion}
\date{\today}

\author
{ Qian Yu
\thanks{ School of Statistics, East China Normal University, Shanghai 200062, China. E-mail: qyumath@163.com.} }

\maketitle

\begin{abstract}
\noindent We consider the existence and H\"{o}lder continuity conditions for the $k$-th order derivatives of self-intersection local time for $d$-dimensional  fractional Brownian motion, where $k=(k_1,k_2,\cdots, k_d)$. Moreover, we show a limit theorem for the critical case with $H=\frac{2}{3}$ and $d=1$, which was conjectured by Jung and Markowsky \cite{Jung 2014}.

\end{abstract}

\section{Introduction}
Fractional Brownian motion (fBm) on $\mathbb{R}^d$  with Hurst parameter $H\in(0,1)$ is a $d$-dimensional centered Gaussian process $B^H=\{B_t^H, ~t\geq0\}$ with covariance function given by
$$
\mathbb{E}\left[B^H_tB^H_s\right]=\frac{1}{2}\left[t^{2H}+s^{2H}-|t-s|^{2H}
\right].
$$
Note that $B_t^{\frac12}$ is a classical standard Brownian motion.
Let $D=\{(r,s): 0<r<s<t\}$. The self-intersection local time (SLT) of fBm was first investigated in Rosen \cite{Rosen 1987} and formally defined as
$$\alpha_t(y)=\int_{D}\delta(B^H_s-B^H_r-y)drds,$$
where $B^H$ is a fBm and $\delta$ is the Dirac delta function. It was further investigated in Hu \cite{Hu 2001}, Hu and Nualart \cite{Hu 2005}.
In particular, Hu and Nualart \cite{Hu 2005} showed its existence whenever $Hd<1$. Moreover,
$\alpha_t(y)$ is H\"{o}lder continuous in time of any order strictly less than $1-H$ which can be derived from Xiao \cite{Xiao 1997}.

The derivative of self-intersection local time (DSLT) for fBm was first considered in the works by Yan et al. \cite{Yan 2008} \cite{Yan 2009}, where the ideas  were based on Rosen \cite{Rosen2005}. The DSLT for fBm has two versions. One is extended by the Tanaka formula (see in Jung and  Markowsky \cite{Jung 2014}):
$$\widetilde{\alpha}'_t(y)=-H\int_{D}\delta'(B^H_s-B^H_r-y)(s-r)^{2H-1}drds.$$

The other is from the occupation-time formula (see Jung and  Markowsky \cite{Jung 2015}):
$$
\widehat{\alpha}'_t(y)=-\int_{D}\delta'(B^H_s-B^H_r-y)drds.
$$

Motivated by the $1$st order DSLT for fBm in Jung and Markowsky \cite{Jung 2015} and the $k$-th order derivative of intersection local time (ILT) for fBm  in Guo et al. \cite{Guo 2017}, we will consider the following $k$-th order DSLT  for fBm in this paper,

\begin{align*}
\widehat{\alpha}^{(k)}(y)&=\frac{\partial^k}{\partial y_1^{k_1}\cdots \partial y_d^{k_d}}\int_{D}\delta(B^H_s-B^H_r-y)drds\\
&=(-1)^{|k|}\int_{D}\delta^{(k)}(B^H_s-B^H_r-y)drds,
\end{align*}
where $k=(k_1,\cdots,k_d)$ is a multi-index with all $k_i$ being nonnegative integers and $|k|=k_1+k_2+\cdots+k_d$, $\delta$ is the Dirac delta function of $d$ variables and $\delta^{(k)}(y)=\frac{\partial^k}{\partial y_1^{k_1}\cdots \partial y_d^{k_d}}\delta(y)$ is the $k$-th order partial derivative of $\delta$.

Set
$$f_\varepsilon(x)=\frac1{(2\pi\varepsilon)^{\frac d2}}e^{-\frac{|x|^2}{2\varepsilon}}=\frac1{(2\pi)^d}\int_{\mathbb{R}^d}e^{i\langle p,x\rangle}e^{-\varepsilon \frac{|p|^2}{2}}dp,$$
where $\langle p,x\rangle=\sum_{j=1}^dp_jx_j$ and $|p|^2=\sum_{j=1}^dp_j^2$.

Since the Dirac delta function $\delta$ can be approximated by $f_\varepsilon(x)$, we approximate $\delta^{(k)}$ and $\widehat{\alpha}_t^{(k)}(y)$ by
$$f^{(k)}_\varepsilon(x)=\frac{i^{|k|}}{(2\pi)^d}\int_{\mathbb{R}^d}p_1^{k_1}\cdots p_d^{k_d}e^{i\langle p,x\rangle}e^{-\varepsilon \frac{|p|^2}{2}}dp$$
and
\begin{equation}\label{sec1-eq1.2}
\widehat{\alpha}^{(k)}_{t,\varepsilon}(y)=(-1)^{|k|}\int_{D}f^{(k)}_\varepsilon(B^H_s-B^H_r-y)drds,
\end{equation}
respectively.

If $\widehat{\alpha}^{(k)}_{t,\varepsilon}(y)$ converges to a random variable in $L^p$ as $\varepsilon\to0$, we denote the limit by $\widehat{\alpha}_t^{(k)}(y)$ and call it the $k$-th DSLT of $B^H$.

\begin{theorem} \label{sec1-th L2}
For $0<H<1$ and $\widehat{\alpha}^{(k)}_{t,\varepsilon}(y)$ defined in \eqref{sec1-eq1.2}, let $\#:=\#\{k_i ~is ~odd, ~i=1, 2, \cdots d\}$ denotes the odd number of $k_i$, for $i=1, 2, \cdots, d$. If $H<\min\{\frac2{2|k|+d},\frac{1}{|k|+d-\#}, \frac1d\}$ for $|k|=\sum_{j=1}^dk_j$, then  $\widehat{\alpha}^{(k)}_{t}(0)$ exists in $L^2$.

\end{theorem}

\begin{theorem} \label{sec1-th Lp}
If $H(|k|+d)<1$, then $\widehat{\alpha}^{(k)}_{t}(0)$ exists in $L^p$, for all $p\in(0,\infty)$.
\end{theorem}

Note that, if $d=1$ and $|k|=1$, the condition for the existence of $\widehat{\alpha}^{(k)}_{t}(y)$ in Theorems \ref{sec1-th L2} and  \ref{sec1-th Lp} are consistent with that in Jung and Markowsky \cite{Jung 2015}. If $d=2$ and $|k|=1$ in Theorems \ref{sec1-th L2}, we can see $H<\frac12$ is the best possible, since a limit theorem for threshold $H=\frac12$ studied in Markowsky \cite{Mar2008}.

\begin{theorem} \label{sec1-th H}
Assume that $H(|k|+d)<1$ and $t, \tilde{t}\in[0,T]$. Then
$\widehat{\alpha}^{(k)}_{t}(y)$ is H\"{o}lder continuous in $y$ of any order strictly less than $\min(1,\frac{1-Hd-H|k|}{H})$
and H\"{o}lder continuous in $t$ of any order strictly  less than $1-H|k|-Hd$,
\begin{equation}\label{sec1-eq1.3}
\left|\mathbb{E}\Big[\Big(\widehat{\alpha}^{(k)}_{t}(x)-\widehat{\alpha}^{(k)}_{t}(y)\Big)^n\Big]\right|\leq C |x-y|^{n\lambda},
\end{equation}
where $\lambda<\min(1,\frac{1-Hd-H|k|}{H})$ and
\begin{equation}\label{sec1-eq1.4}
\left|\mathbb{E}\Big[\Big(\widehat{\alpha}^{(k)}_{t}(y)-\widehat{\alpha}^{(k)}_{\tilde{t}}(y)\Big)^n\Big]\right|\leq C |t-\tilde{t}|^{n\beta},
\end{equation}
where $\beta<1-H|k|-Hd$.

\end{theorem}

Note that, if $d=1$ and $k=1$. The results of \eqref{sec1-eq1.3} and \eqref{sec1-eq1.4} in Theorem \ref{sec1-th H} are consistent with the results in Jung and Markowsky \cite{Jung 2015}. When $d=1$ and $k=0$, the corresponding H\"{o}lder continuous in time of any order less than $1-H$, which is the condition obtained in Xiao \cite{Xiao 1997}. Moreover, we believe that our methodology also works well for $k$-th order DSLT of stochastic differential equation (SDE) driven by fBm, if the solution of SDE driven by fBm satisfy the property of local nondeterminism. For example, the special linear SDE, i.e. fractional Ornstein-Uhlenbeck processes.

Jung and Markowsky \cite{Jung 2014} proved that $\widehat{\alpha}^{(k)}_{t}(0)$ exists in $L^2$ for $d=1$ and $k=1$ with $0<H<2/3$, and conjectured that for the case $H>2/3$, $\varepsilon^{-\gamma(H)}\widehat{\alpha}^{'}_{t,\varepsilon}(0)$ converges in law to a Gaussian distribution for some suitable constant $\gamma(H)>0$, and at the critical point $H=\frac23$, the variable $\log(\frac1\varepsilon)^{-\gamma}\widehat{\alpha}^{'}_{t,\varepsilon}(0)$ converges in law to a Gaussian distribution for some $\gamma>0$. Later, Jaramillo and Nualart \cite{Jaramillo 2017}  proved the case of $H>2/3$ as
$$\varepsilon^{\frac32-\frac1H}\widehat{\alpha}^{'}_{t,\varepsilon}(0)\overset{law}{\to}N(0,\sigma_0^2), ~~\varepsilon\to0.$$

 By the proof of Lemma \ref{sec3-lem int} in Section 2, we can see the multinomial terms $(p_{i1}-\frac{\mu p_{i2}}{\rho})^{k_i}$ for $i=1,2,...,d$, are taken into account. But we are not sure if $k_i$ is odd or even, there are many difficulties in the integral of $\int(p_{i1}-\frac{\mu p_{i2}}{\rho})^{k_i}e^{-\frac{\rho p^2_{i1}}{2}}dp_{i1}$, thus we only consider the case $d=1$ and $k=1$ below.

Inspired by the results conjectured in \cite{Jung 2014} and the functional limit theorem for SLT of fBm given in Jaramillo and Nualart \cite{Jaramillo 2019}. We will show a limit theorem of the  critical case $H=\frac{2}{3}$.

\begin{theorem} \label{sec1-th sup}
For $\widehat{\alpha}^{(k)}_{t,\varepsilon}(y)$  defined in \eqref{sec1-eq1.2} with $y=0$. Suppose that $H=\frac{2}{3}$, $d=1$ and $k=1$,
then as $\varepsilon\to0$, we have
$$
\Big(\log1/\varepsilon\Big)^{-1}\widehat{\alpha}^{'}_{t,\varepsilon}(0)\overset{law}{\to}N(0,\sigma^2),
$$
where
$\sigma^2=\frac{t^{\frac43}}{8\pi }B(2,1/3)$ and $B(\cdot,\cdot)$ is a Beta function.
\end{theorem}

The study of DSLT for fBm has a strong degree of heat, see in \cite{Jaramillo 2017}-\cite{Jung 2015}, \cite{Yan 2015} and references therein. However, the corresponding results for higher order derivative have not been studied, except for the higher order derivative of ILT for two independent fBms and some general Gaussian processes in \cite{Guo 2017} and \cite{HX}. As we all know, SLT and ILT have different integral structures in form. In particular, the independence of two fBms is required for ILT. So that the nondeterminism property which used for higher order derivative of ILT  can not be used directly here.

To obtain the main results, we would use the methods of sample configuration given in Jung and Markowsky \cite{Jung 2015} and  chaos decomposition provided in Jaramillo and Nualart \cite{Jaramillo 2017}. Chaos decomposition is more and more mature for the asymptotic properties of SLT (see in Hu \cite{Hu 2005}, Jaramillo and Nualart \cite{Jaramillo 2017} and the references therein). The sample configuration method gives a way to apply  nondeterminism property and
it is very powerful to prove the  H\"{o}lder regularity.
But the corresponding results of higher order DSLT for $d$-dimensional fBm still has certain difficulty. The main difficulty lies in the computational complexity of multiple stochastic integrals. Moreover, the related results can be extended to the general cases.
By the Theorem 4.1 in Jaramillo and Nualart \cite{Jaramillo 2017} and Theorem \ref{sec1-th sup} here, two limit theorems of the case $H>\frac{2}{2|k|+d}$ and critical case $H=\frac{2}{2|k|+d}$, with general $k=(k_1,\cdots, k_d)$ are left open. Extending these limit theorems to general cases will be worked in the future.

The paper has the following structure. We present some preliminary properties of $d$-dimensional fBm and some basic lemmas in Section 2. Section 3 is to prove the main results. To be exact, we will split this section into four subsections to prove the four theorems given in Section 1.
Throughout this paper, if not mentioned otherwise, the letter $C$, with or without a subscript,
denotes a generic positive finite constant and may change from line to line.

\section{Preliminaries}
In this section, we first give some properties of $d$-dimensional fBm $B^H$. It is well known that
$d$-dimensional fBm has self-similarity, stationary increments and H\"{o}lder continuity. When Hurst parameter $H>1/2$, $B^H$ exhibits long memory. When $H<1/2$, it has short memory. But in this paper, we need the following nondeterminism property.

By Nualart and Xu \cite{Nualart2014} (see also in Song, Xu and Yu \cite{Song 2018}), we can see that for any $n\in\mathbb{N}$, there exists two constants $\kappa_{H}$ and $\beta_{H}$ depending only on  $n$ and $H$, such that for any $0=s_0<s_1 <\cdots < s_n $, $1\le i \le n$, we have
\[
\kappa_{H} \sum_{i=1}^n |x_i|^2 (s_i -s_{i-1})^{2H} \le \mathrm{Var} \Big( \sum_{i=1}^n x_i  \cdot (B^H_{s_i} -B^H_{s_{i-1}}) \Big) \le \beta_{H}  \sum_{i=1}^n |x_i|^2 (s_i -s_{i-1})^{2H}.
\]

Next, we present two basic lemmas, which will be used in Section 3.

\begin{lemma} \label{sec3-lem int}
For any $\lambda, ~\mu, ~\rho\in \mathbb{R}$ with $\lambda>0$, $\rho>0$ and $\lambda\rho-\mu^2>0$. For $k\in\mathbb{Z}^+$, there exists a constant $C$ only depending on $k$,  such that

(i) if $k$ is odd,
\begin{equation}\label{sec3-eq3.1}
\begin{split}
\Big|\int_{\mathbb{R}^2}x^ky^ke^{-\frac12(\lambda x^2+\rho y^2+2\mu xy)}dxdy\Big|\leq
\begin{cases}
    \frac{C |\mu|^k}{(\lambda\rho-\mu^2)^{k+\frac12}}, &\mbox{if $\frac{\mu^2}{\lambda\rho-\mu^2}\geq1$,}\\
    \frac{C |\mu|}{(\lambda\rho-\mu^2)^{\frac{k}2+1}}, &\mbox{if $\frac{\mu^2}{\lambda\rho-\mu^2}<1$,}
   \end{cases}
\end{split}
\end{equation}

(ii) if $k$ is even,
\begin{equation}\label{sec3-eq3.2}
\begin{split}
\Big|\int_{\mathbb{R}^2}x^ky^ke^{-\frac12(\lambda x^2+\rho y^2+2\mu xy)}dxdy\Big|\leq
\begin{cases}
    \frac{C |\mu|^k}{(\lambda\rho-\mu^2)^{k+\frac12}}, &\mbox{if $\frac{\mu^2}{\lambda\rho-\mu^2}\geq1$,}\\
    \frac{C }{(\lambda\rho-\mu^2)^{\frac{k+1}2}}, &\mbox{if $\frac{\mu^2}{\lambda\rho-\mu^2}<1$.}
   \end{cases}
\end{split}
\end{equation}

\end{lemma}

\begin{proof}
First, we consider the integral with respect to $y$,
\begin{align*}
\int_{\mathbb{R}}y^ke^{-\frac{\rho}{2}y^2-\mu xy}dy&=e^{\frac{\mu^2x^2}{2\rho}}\int_{\mathbb{R}}y^ke^{-\frac{\rho}{2}(y+\frac{\mu x}{\rho})^2}dy\\
&=e^{\frac{\mu^2x^2}{2\rho}}\int_{\mathbb{R}}(y-\frac{\mu x}{\rho})^ke^{-\frac{\rho}{2}y^2}dy.
\end{align*}

If $k$ is odd, since
$$
(y-\frac{\mu x}{\rho})^k=\sum_{i=0}^kC_k^iy^i(-\frac{\mu x}{\rho})^{k-i},
$$
we have

\begin{align*}
\int_{\mathbb{R}}y^ke^{-\frac{\rho}{2}y^2-\mu xy}dy
&=e^{\frac{\mu^2x^2}{2\rho}}\int_{\mathbb{R}}A_{odd}e^{-\frac{\rho}{2}y^2}dy\\
&=C_1e^{\frac{\mu^2x^2}{2\rho}}\frac{(\mu x)^k}{\rho^{k+\frac12}}+C_3e^{\frac{\mu^2x^2}{2\rho}}\frac{(\mu x)^{k-2}}{\rho^{k-\frac12}}+\cdots+C_{k}e^{\frac{\mu^2x^2}{2\rho}}\frac{(\mu x)}{\rho^{k/2+1}}\\
&=:\widetilde{A}_{odd},
\end{align*}
where $C_1, ~C_3, \cdots, ~C_k$ are all positive constants and
$$A_{odd}=C_k^0(-\frac{\mu x}{\rho})^{k}+C_k^2y^2(-\frac{\mu x}{\rho})^{k-2}+\cdots+C_k^{k-1}y^{k-1}(-\frac{\mu x}{\rho}).$$

For the $dx$ integral,
$$\int_{\mathbb{R}}x^ke^{-\frac12\lambda x^2}(\widetilde{A}_{odd})dx\leq C\left[\frac{\mu^k}{(\lambda\rho-\mu^2)^{k+\frac12}}+\frac{\mu^{k-2}}{(\lambda\rho-\mu^2)^{k-\frac12}}+\cdots+\frac{\mu}{(\lambda\rho-\mu^2)^{\frac{k}{2}+1}}\right],
$$
where the right hand side is the sum of equal ratio series with the common ratio $\frac{\mu^2}{\lambda\rho-\mu^2}>0$.
Then, we get  \eqref{sec3-eq3.1}.

If $k$ is even,
\begin{align*}
\int_{\mathbb{R}}y^ke^{-\frac{\rho}{2}y^2-\mu xy}dy&=e^{\frac{\mu^2x^2}{2\rho}}\int_{\mathbb{R}}(y-\frac{\mu x}{\rho})^ke^{-\frac{\rho}{2}y^2}dy\\
&\leq C e^{\frac{\mu^2x^2}{2\rho}}\Big[\int_{\mathbb{R}}y^ke^{-\frac{\rho}{2}y^2}dy+\int_{\mathbb{R}}(\frac{\mu x}{\rho})^ke^{-\frac{\rho}{2}y^2}dy\Big]\\
&=:B_1+B_2.
\end{align*}

It is easy to see that
$$B_1\leq C e^{\frac{\mu^2x^2}{2\rho}}\rho^{-\frac{k+1}{2}}$$
and
$$B_2\leq C e^{\frac{\mu^2x^2}{2\rho}}(\mu x)^k\rho^{-\frac{2k+1}{2}}.$$

For the integral with respect to $x$,
\begin{align*}
\frac{1}{\rho^{\frac{k+1}2}}\int_{\mathbb{R}}x^{k}e^{-\frac{x^2}{2\rho}(\lambda\rho-\mu^2)}dx
&\leq \frac{C}{\rho^{\frac{k+1}2}}\int_{\mathbb{R}}x^{k}e^{-\frac{x^2}{2}}\Big(\frac{\lambda\rho-\mu^2}{\rho}\Big)^{-\frac{1+k}{2}}dx\\
&\leq \frac{C}{(\lambda\rho-\mu^2)^{\frac{k+1}2}}
\end{align*}
and
\begin{align*}
\frac{\mu^k}{\rho^{k+\frac12}}\int_{\mathbb{R}}x^{2k}e^{-\frac{x^2}{2\rho}(\lambda\rho-\mu^2)}dx
&\leq C\frac{\mu^k}{\rho^{k+\frac12}}\int_{\mathbb{R}}x^{2k}e^{-\frac{x^2}{2}}\Big(\frac{\lambda\rho-\mu^2}{\rho}\Big)^{-\frac{1+2k}{2}}dx\\
&\leq C \frac{\mu^k}{(\lambda\rho-\mu^2)^{k+\frac12}}.
\end{align*}

This gives \eqref{sec3-eq3.2}.
\end{proof}

The next lemma gives the bounds on the quantity of $\lambda\rho-\mu^2$, which could be obtained from the Appendix B in \cite{Jung 2014} or the Lemma 3.1 in \cite{Hu 2001}.

\begin{lemma} \label{sec3-lem3.2}
Let $$\lambda=|s-r|^{2H}, ~~\rho=|s'-r'|^{2H},$$
and
$$\mu=\frac12\Big(|s'-r|^{2H}+|s-r'|^{2H}-|s'-s|^{2H}-|r-r'|^{2H}\Big).$$

\textbf{Case (i)} Suppose that $D_1=\{(r,r',s,s')\in[0,t]^4 ~|~ r<r'<s<s'\}$, let $r'-r=a$, $s-r'=b$, $s'-s=c$. Then, there exists a constant $K_1$ such that
$$\lambda\rho-\mu^2\geq K_1\,\left((a+b)^{2H}c^{2H}+a^{2H}(b+c)^{2H}\right)$$
and
$$2\mu=(a+b+c)^{2H}+b^{2H}-a^{2H}-c^{2H}.$$

\textbf{Case (ii)} Suppose that $D_2=\{(r,r',s,s')\in[0,t]^4 ~|~ r<r'<s'<s\}$, let $r'-r=a$, $s'-r'=b$, $s-s'=c$. Then, there exists a constant $K_2$ such that
$$\lambda\rho-\mu^2\geq K_2\,b^{2H}\left(a^{2H}+c^{2H}\right)$$
and
$$2\mu=(a+b)^{2H}+(b+c)^{2H}-a^{2H}-c^{2H}.$$

\textbf{Case (iii)} Suppose that $D_3=\{(r,r',s,s')\in[0,t]^4 ~|~ r<s<r'<s'\}$, let $s-r=a$, $r'-s=b$, $s'-r'=c$. Then, there exists a constant $K_3$ such that
$$\lambda\rho-\mu^2\geq K_3(ac)^{2H}$$
and
$$2\mu=(a+b+c)^{2H}+b^{2H}-(a+b)^{2H}-(c+b)^{2H}.$$
\end{lemma}

\section{Proof of the main results}

In this section, the proof of  Theorems \ref{sec1-th L2}, \ref{sec1-th Lp}, \ref{sec1-th H} and \ref{sec1-th sup}  are taken into account. We will divide this section into four parts and give the proof of the corresponding theorem in each part.

\subsection{Proof of Theorem \ref{sec1-th L2}}

By \eqref{sec1-eq1.2} and the proof of Lemma \ref{sec3-lem int},
\begin{align*}
\mathbb{E}\Big[\widehat{\alpha}^{(k)}_{t,\varepsilon}(0)\widehat{\alpha}^{(k)}_{t,\eta}(0)\Big]&=\frac1{(2\pi)^{2d}}\int_{D^2}\int_{\mathbb{R}^{2d}}p_1^kp_2^ke^{\frac{-(\varepsilon|p_1|^2+\eta|p_2|^2)}{2}}\mathbb{E}\Big[\prod_{j=1}^2e^{i\langle p_j, B^H_{s_j}-B^H_{r_j}\rangle}\Big]dp_1dp_2 dr'dr ds'ds\\
&=C \int_{D^2}\int_{\mathbb{R}^{2d}}\prod_{i=1}^dp_{i1}^{k_i}\prod_{i=1}^dp_{i2}^{k_i}e^{-\frac12(|p_1|^2(\lambda+\varepsilon)+|p_2|^2(\rho+\eta)+2\langle p_1, p_2\rangle\mu)}dp_1dp_2 dr'dr ds'ds\\
&=C \int_{D^2}\Big[\int_{\mathbb{R}^{2}}p_{11}^{k_1}p_{12}^{k_1}e^{-\frac12(p_{11}^2(\lambda+\varepsilon)+p_{12}^2(\rho+\eta)+2p_{11}p_{12}\mu)}dp_{11}dp_{12} \Big]\\
&\qquad\times \cdots \times \Big[\int_{\mathbb{R}^{2}}p_{d1}^{k_d}p_{d2}^{k_d}e^{-\frac12(p_{d1}^2(\lambda+\varepsilon)+p_{d2}^2(\rho+\eta)+2p_{d1}p_{d2}\mu)}dp_{d1}dp_{d2} \Big]dr'dr ds'ds\\
&=:C\int_{D^2}\prod_{i=1}^d\Xi_{k_i}dr'dr ds'ds,
\end{align*}
where
$\lambda=|s-r|^{2H}, ~\rho=|s'-r'|^{2H}$,
$$\mu=\frac12\Big(|s'-r|^{2H}+|s-r'|^{2H}-|s'-s|^{2H}-|r-r'|^{2H}\Big)$$
and
\begin{align*}
\Xi_{k_i}=
\begin{cases}
    \frac{C_{k_i} \mu^{k_i}}{((\lambda+\varepsilon)(\rho+\eta)-\mu^2)^{k_i+\frac12}}
    +\frac{C_{k_i-2} \mu^{k_i-2}}{((\lambda+\varepsilon)(\rho+\eta)-\mu^2)^{k_i-\frac12}}
    +\cdots+\frac{C_{1} \mu}{((\lambda+\varepsilon)(\rho+\eta)-\mu^2)^{\frac{k_i}{2}+1}}, &\mbox{if $k_i$ is odd,}\\
    \frac{C_{k_i} \mu^{k_i}}{((\lambda+\varepsilon)(\rho+\eta)-\mu^2)^{k_i+\frac12}}
    +\frac{C_{k_i-2} \mu^{k_i-2}}{((\lambda+\varepsilon)(\rho+\eta)-\mu^2)^{k_i-\frac12}}
    +\cdots+\frac{C_{0}}{((\lambda+\varepsilon)(\rho+\eta)-\mu^2)^{\frac{k_i}{2}+\frac12}}, &\mbox{if $k_i$ is even,}
   \end{cases}
\end{align*}
for $i=1, 2, \cdots, d$.

Notice that, for any $\varepsilon_1, ~\varepsilon_2>0$,
\begin{align*}
\mathbb{E}\Big[\big(\widehat{\alpha}^{(k)}_{t,\varepsilon_1}(0)-\widehat{\alpha}^{(k)}_{t,\varepsilon_2}(0)\big)^2\Big]&\leq C \int_{D^2}\bigg|\int_{\mathbb{R}^{2d}}\prod_{j=1}^2\left(e^{-\frac{\varepsilon_1}{2}|p_j|^2}-e^{-\frac{\varepsilon_2}{2}|p_j|^2}\right)\\
&\qquad\qquad\qquad\times p_1^kp_2^k\mathbb{E}\Big[\prod_{j=1}^2e^{i\langle p_j, B^H_{s_j}-B^H_{r_j}\rangle}\Big]dp_1dp_2 \bigg|dr'drds'ds\\
&\leq C\int_{D^2}\prod_{j=1}^2\max_{p_j}\left|e^{-\frac{\varepsilon_1}{2}|p_j|^2}-e^{-\frac{\varepsilon_2}{2}|p_j|^2}\right|\prod_{i=1}^d\left|\widetilde{\Xi_{k_i}}\right|dr'drds'ds,
\end{align*}
where
\begin{align*}
\widetilde{\Xi_{k_i}}&=
\begin{cases}
    \frac{C_{k_i} \mu^{k_i}}{(\lambda\rho-\mu^2)^{k_i+\frac12}}
    +\frac{C_{k_i-2} \mu^{k_i-2}}{(\lambda\rho-\mu^2)^{k_i-\frac12}}
    +\cdots+\frac{C_{1} \mu}{(\lambda\rho-\mu^2)^{\frac{k_i}{2}+1}}, &\mbox{if $k_i$ is odd,}\\
    \frac{C_{k_i} \mu^{k_i}}{(\lambda\rho-\mu^2)^{k_i+\frac12}}
    +\frac{C_{k_i-2} \mu^{k_i-2}}{(\lambda\rho-\mu^2)^{k_i-\frac12}}
    +\cdots+\frac{C_{0}}{(\lambda\rho-\mu^2)^{\frac{k_i}{2}+\frac12}}, &\mbox{if $k_i$ is even.}
   \end{cases}
\end{align*}

Consequently, if
$$\int_{D^2}\prod_{i=1}^d\left|\widetilde{\Xi_{k_i}}\right|dr'dr ds'ds<\infty,$$
then $\widehat{\alpha}^{(k)}_{t,\varepsilon}(0)$ converges in $L^2$ as $\varepsilon\to0$.

By Lemma \ref{sec3-lem int}, we can see that $\int_{D^2}\prod_{i=1}^d\left|\widetilde{\Xi_{k_i}}\right|dr'dr ds'ds$
is less than
\begin{align*}
C \int_{D^2} \left[\frac{|\mu|^{|k|}}{(\lambda\rho-\mu^2)^{|k|+\frac{d}2}}+\frac{|\mu|^{\#}}{(\lambda\rho-\mu^2)^{\frac{|k|+d+\#}2}}\right]dr'dr ds'ds,
\end{align*}
where $\#=\#\{k_i ~is ~odd, ~i=1, 2, \cdots, d\}$ denotes the odd number of $k_i$, for $i=1, 2, \cdots, d$, and $\#\in\{0, 1, 2, \cdots, d\}$.

Thus, to prove the finiteness of $\int_{D^2}\prod_{i=1}^d\left|\widetilde{\Xi_{k_i}}\right|dr'dr ds'ds$, we only need to prove
\begin{equation}\label{sec3-eq3.3}
\int_{D^2} \frac{|\mu|^{Q}}{(\lambda\rho-\mu^2)^{\frac{|k|+d+Q}2}}dr'dr ds'ds<\infty
\end{equation}
with $Q=|k|$ and $Q=\#$.

By Lemma \ref{sec3-lem3.2}, we can see $D^2$ is the union of the sets $D_1, ~D_2, ~D_3$.

When $(r,r',s,s')\in D_1$,  then the left hand side of \eqref{sec3-eq3.3} is less than
\begin{align*}
C &\int_{[0,t]^3}\frac{a^{2HQ}+b^{2HQ}+c^{2HQ}}{(a+b)^{\frac{H}{2}(|k|+d+Q)}(b+c)^{\frac{H}{2}(|k|+d+Q)}(ac)^{\frac{H}{2}(|k|+d+Q)}}dadbdc\\
&\leq C\int_{[0,t]^3}\frac{a^{2HQ}}{a^{\frac{H}{2}(|k|+d+Q)}b^{\frac{H}{2}(|k|+d+Q)}(ac)^{\frac{H}{2}(|k|+d+Q)}}dadbdc\\
&\quad+C\int_{[0,t]^3}\frac{b^{2HQ}}{b^{\frac{H}{2}(|k|+d+Q)}b^{\frac{H}{2}(|k|+d+Q)}(ac)^{\frac{H}{2}(|k|+d+Q)}}dadbdc\\
&\quad+C\int_{[0,t]^3}\frac{c^{2HQ}}{b^{\frac{H}{2}(|k|+d+Q)}c^{\frac{H}{2}(|k|+d+Q)}(ac)^{\frac{H}{2}(|k|+d+Q)}}dadbdc\\
&\leq C\int_{[0,t]^3}\frac{1}{x^{H(|k|+d-Q)}y^{\frac{H}{2}(|k|+d+Q)}z^{\frac{H}{2}(|k|+d+Q)}}dxdydz<\infty.
\end{align*}

When $(r,r',s,s')\in D_2$.
Note that, for the condition $H<\min\{\frac2{2|k|+d},\frac{1}{|k|+d-\#}, \frac1d\}$, only in the case $d=1$, $H$ can get the value bigger than $\frac12$, while these the case have been studied in \cite{Hu 2005} and \cite{Jung 2014}, respectively. So, we only need to consider the case $H<\frac12$.

Thus, the left hand side of \eqref{sec3-eq3.3} is less than
\begin{align*}
C &\int_{[0,t]^3}\frac{[(a+b)^{2H}+(b+c)^{2H}-a^{2H}-c^{2H}]^{Q}}{b^{H(|k|+d+Q)}(ac)^{\frac{H}{2}(|k|+d+Q)}}dadbdc\\
&\leq C\int_{[0,t]^3}\frac{[(a+b)^{2H}-a^{2H}]^{Q}}{b^{H(|k|+d+Q)}(ac)^{\frac{H}{2}(|k|+d+Q)}}dadbdc\\
&\quad\quad+ C\int_{[0,t]^3}\frac{[(b+c)^{2H}-c^{2H}]^{Q}}{b^{H(|k|+d+Q)}(ac)^{\frac{H}{2}(|k|+d+Q)}}dadbdc\\
&\leq 2C\int_{[0,t]^2}\frac{[(a+b)^{2H}-a^{2H}]^{Q}}{b^{H(|k|+d+Q)}a^{\frac{H}{2}(|k|+d+Q)}}dadb.
\end{align*}

Since $H<\frac12$, then
$$[(a+b)^{2H}-a^{2H}]^{Q}\leq b^{2HQ}.$$

Thus,
\begin{align*}
\int_{[0,t]^2}\frac{[(a+b)^{2H}-a^{2H}]^{Q}}{b^{H(|k|+d+Q)}a^{\frac{H}{2}(|k|+d+Q)}}dadb
&\leq C \int_{[0,t]^2}\frac{b^{2HQ}}{b^{H(|k|+d+Q)}a^{\frac{H}{2}(|k|+d+Q)}}dadb\\
&\leq C \int_{[0,t]^2}\frac{1}{b^{H(|k|+d-Q)}a^{\frac{H}{2}(|k|+d+Q)}}dadb<\infty.
\end{align*}

When $(r,r',s,s')\in D_3$.
For $\alpha, \beta>0$ with $\alpha+\beta=1$, there exists a positive constant $K$ such that
\begin{align*}
|\mu|&=\frac12\Big|(a+b+c)^{2H}+b^{2H}-(a+b)^{2H}-(b+c)^{2H}\Big|\\
&=\Big|H(2H-1)ac\int_0^1\int_0^1(b+au+cv)^{2H-2}dudv\Big|\\
&\leq ac\int_0^1\int_0^1\Big[b^\alpha(au+cv)^\beta\Big]^{2H-2}dudv\\
&\leq ac\int_0^1\int_0^1\Big[b^\alpha(au)^{\frac{\beta}{2}}(cv)^{\frac{\beta}{2}}\Big]^{2H-2}dudv\\
&\leq K (ac)^{\beta(H-1)+1}b^{2\alpha(H-1)}.
\end{align*}

Thus, the left hand side of \eqref{sec3-eq3.3} is less than
\begin{align*}
C \int_{[0,t]^3}\frac{[(ac)^{\beta(H-1)+1}b^{2\alpha(H-1)}]^{Q}}{(ac)^{H(|k|+d+Q)}}dadbdc
&\leq C\int_{[0,t]^3}\frac{1}{b^{2\alpha Q(1-H)}(ac)^{\beta(Q-HQ)+H|k|+Hd+HQ-Q}}dadbdc.
\end{align*}

Note that $|k|\geq 1$ (where $|k|=0$ with $H<\frac1d$ could deduced from \cite{Hu 2005}) and $H<\frac12$.
When $Q=0$ (all derivatives were of even order), the result of \eqref{sec3-eq3.3} is obvious by $H<\frac1{|k|+d}$.
When $Q\geq1$, we have $2Q(1-H)>1$. So, we first choose $\varepsilon_0>0$, such that
$$H(|k|+d)-\frac12+\frac{\varepsilon_0}{2}\Big(\frac2{|k|+d+Q}-H\Big)<1.$$
Then we can choose
$$\alpha\in\left(\frac{1-\varepsilon_0(\frac2{|k|+d+Q}-H)}{2Q(1-H)},\frac{1}{2Q(1-H)}\right).$$
Thus,
\begin{align*}
\beta(Q-HQ)+H|k|+Hd+HQ-Q&=(1-\alpha)(Q-HQ)+H|k|+Hd+HQ-Q\\
&<H(|k|+d)-\frac12+\frac{\varepsilon_0}{2}\Big(\frac2{|k|+d+Q}-H\Big),
\end{align*}
which is less than one. This gives \eqref{sec3-eq3.3}.

\subsection{Proof of Theorem \ref{sec1-th Lp}}

By \eqref{sec1-eq1.2}, we have
\begin{align*}
\left|\mathbb{E}\Big[\big(\widehat{\alpha}^{(k)}_{t,\varepsilon}(0)\big)^n\Big]\right|&\leq C \int_{D^n}\int_{\mathbb{R}^{nd}}\prod_{j=1}^n|p_j^k|\mathbb{E}\Big[\prod_{j=1}^ne^{i\langle p_j, B^H_{s_j}-B^H_{r_j}\rangle}\Big]dp dr ds,
\end{align*}
where $k=(k_1, \cdots, k_d)$, $|p_j^k|=\prod_{i=1}^d|p_{ij}|^{k_i}$ for $j=1,...,n$, $drds=dr_1\cdots dr_nds_1\cdots ds_n$
and
$$dp=dp_1\cdots dp_n=dp_{11}dp_{12}\cdots dp_{1n}\cdots dp_{d1}dp_{d2}\cdots dp_{dn}.$$

We use the method of sample configuration as in Jung and Markowsky \cite{Jung 2015}.
Fix an ordering of the set $\{r_1, s_1, r_2, s_2, \cdots, r_n, s_n\}$, and let $l_1\leq l_2\leq \cdots \leq l_{2n}$ be a relabeling of the set $\{r_1, s_1, r_2, s_2, \cdots, r_n, s_n\}$. Let $u_1 \ldots u_{2n-1}$ be the proper linear combinations of the $p_{j}$'s so that

$$
\mathbb{E}\Big[\prod_{j=1}^n e^{i\langle p_j, B^H_{s_j}-B^H_{r_j}\rangle}\Big] = \mathbb{E}\Big[\prod_{j=1}^{2n-1}e^{i\langle u_j, B^H_{l_{j+1}}-B^H_{l_j}\rangle}\Big].
$$

A detailed description of how the $u$'s are chosen can be found in \cite{Jung 2015}.
Then by the local nondeterminism of fBm,
$$\Big|\mathbb{E}\Big[\prod_{j=1}^ne^{i\langle p_j, B^H_{s_j}-B^H_{r_j}\rangle}\Big]\Big|\leq e^{-c \sum_{j=1}^{2n-1}|u_j|^2(l_{j+1}-l_j)^{2H}}.$$

Fix $j$, and let $j_1$ to be the smallest value such that $u_{j_1}$ contains $p_{j}$ as a term and then choose $j_2$ to be
the smallest value strictly larger than $j_1$ such that $u_{j_2}$ does not contain $p_{j}$ as a term. Then $p_j = u_{j_1} - u_{j_1-1} = u_{j_2-1} - u_{j_2}$.
Similarly to Jung and Markowsky \cite{Jung 2015}, we can see that, with the convention that $u_{0}=u_{2n}=0$,
\begin{align*}
|p_j^k|=&|(u_{j_1}-u_{j_1-1})^{\frac {k}2}||(u_{j_2-1}-u_{j_2})^{\frac {k}2}| \\
&=\prod_{i=1}^d|u_{ij_1}-u_{i(j_1-1)}|^{\frac {k_i}2}|u_{i(j_2-1)}-u_{ij_2}|^{\frac {k_i}2} \\
&\leq C\prod_{i=1}^d(|u_{ij_1}|^{\frac {k_i}2}+|u_{i(j_1-1)}|^{\frac {k_i}2})(|u_{i(j_2-1)}|^{\frac {k_i}2}+|u_{ij_2}|^{\frac {k_i}2})
\end{align*}

Thus,

$$
\prod_{j=1}^n|p_j^k| = \prod_{j=1}^{2n}|(u_{j}-u_{j-1})^{\frac {k}2}| \leq C \prod_{i=1}^d \prod_{j=1}^{2n} (|u_{ij}|^{\frac {k_i}2}+|u_{i(j-1)}|^{\frac {k_i}2}).
$$
and
\begin{equation} \label{tria}
\begin{split}
\prod_{i=1}^d\prod_{j=1}^{2n}(|u_{ij}|^{\frac {k_i}2}+|u_{i(j-1)}|^{\frac {k_i}2})
&=\sum_{S_1}\prod_{i=1}^d\prod_{j=1}^{2n}(|u_{ij}|^{\frac {k_i}2\gamma_{i,j}}|u_{i(j-1)}|^{\frac {k_i}2\overline{\gamma_{i,j}}})\\
&\leq\sum_{S_2}\prod_{i=1}^d\prod_{j=1}^{2n-1}(|u_{ij}|^{\frac {k_i}2\alpha_{i,j}}),
\end{split}
\end{equation}
where $$S_1=\left\{\gamma_{i,j}, ~\overline{\gamma_{i,j}}: ~\gamma_{i,j}\in\{0,1\}, ~\gamma_{i,j}+\overline{\gamma_{i,j}}=1, ~~i=1,\cdots, d, ~j=1, \cdots, 2n\right\}$$
and
$$
S_2=\left\{\alpha_{i,j}: ~\alpha_{i,j}\in\{0, 1, 2\}, ~i=1,\cdots, d, ~j=1, \cdots 2n-1\right\}.
$$

Note that we have omitted the terms $j=0, 2n$ in the final expression in (\ref{tria}) since $u_0=u_{2n} = 0$. Then
\begin{align*}
\left|\mathbb{E}\Big[(\widehat{\alpha}^{(k)}_{t,\varepsilon}(0))^n\Big]\right|
&\leq C \int_{E^n}\int_{\mathbb{R}^{nd}}e^{-c \sum_{j=1}^{2n-1}|u_j|^2(l_{j+1}-l_j)^{2H}}\prod_{i=1}^d\prod_{j=1}^{2n}(|u_{ij}|^{\frac {k_i}2}+|u_{i(j-1)}|^{\frac {k_i}2})dp dl\\
&\leq C \sum_{S_2}\int_{E^n}\int_{\mathbb{R}^{nd}}e^{-c \sum_{j=1}^{2n-1}|u_j|^2(l_{j+1}-l_j)^{2H}}\prod_{i=1}^d\prod_{j=1}^{2n-1}(|u_{ij}|^{\frac {k_i}{2}\alpha_{i,j}})dp dl\\
&=C \sum_{S_2}\int_{E^n}\int_{\mathbb{R}^{nd}}e^{-c \sum_{j=1}^{2n-1}|u_j|^2(l_{j+1}-l_j)^{2H}}\prod_{j=1}^{2n-1}(|u_j^{\frac {k}2\alpha_j}|)dp dl,
\end{align*}
where $E^n=\{0<l_1<\cdots<l_{2n}<t\}$, $|u_j^{\frac {k}2\alpha_j}|=\prod_{i=1}^d|u_{ij}|^{\frac{k_i}{2}\alpha_{i,j}}$ and
$dl=dl_1dl_2\cdots dl_{2n}.$

It is easy to observe that $\{u_{1}, u_{2},\cdots, u_{2n-1}\}$ is contained in the span of $\{p_{1}, p_{2}, \cdots, p_{n}\}$ and conversely, so we can let $\mathcal{A}$ be a subset of $\{1, \ldots , 2n-1\}$ such that the set $\{u_j\}_{j \in \mathcal{A}}$ spans $\{p_{1}, p_{2}, \cdots, p_{n}\}.$ We let $\mathcal{A}^c$ denote the complement of $\mathcal{A}$ in $\{1, \ldots , 2n-1\}$. Note that

\begin{align*}
&e^{-c \sum_{j\in \AA^c}|u_{j}|^2(l_{j+1}-l_j)^{2H}}\prod_{j\in\AA^c}(|u_{j}^{\frac {k}2\alpha_j}|)\\
&\quad=e^{-c \sum_{j\in \AA^c}|u_{j}|^2(l_{j+1}-l_j)^{2H}}\prod_{j\in \AA^c}\Big(|u_{j}^{\frac k2\alpha_j}|(l_{j+1}-l_j)^{\frac{H|k\alpha_j|}{2}}\Big)\prod_{j\in \AA^c}(l_{j+1}-l_j)^{-\frac{H|k\alpha_j|}{2}}\\
&\quad\leq C\prod_{j\in \AA^c}(l_{j+1}-l_j)^{-\frac{H|k\alpha_j|}{2}},
\end{align*}
where $|k\alpha_j|=k_1\alpha_{1,j}+\cdots+k_d\alpha_{d,j}$.
Then, we perform a linear transformation changing $(p_{1}, ~p_{2}, \cdots, p_{n})$ into an integral with respect to variables $\{u_{j}\}_{j \in \AA}$,
\begin{align*}
&\int_{\mathbb{R}^{nd}}e^{-c \sum_{j=1}^{2n-1}|u_j|^2(l_{j+1}-l_j)^{2H}}\prod_{j=1}^{2n-1}(|u_j^{\frac {k}{2}\alpha_j}|)dp\\
&\qquad\leq C\prod_{j\in \AA^c}(l_{j+1}-l_j)^{-\frac{H|k\alpha_j|}{2}}\int_{\mathbb{R}^{nd}}e^{-c \sum_{j\in \AA}|u_{j}|^2(l_{j+1}-l_j)^{2H}}\prod_{j\in \AA}(|u_{j}^{\frac {k}2\alpha_j}|)dp\\
&\qquad=C|J|\prod_{j\in \AA^c}(l_{j+1}-l_j)^{-\frac{H|k\alpha_j|}{2}}\int_{\mathbb{R}^{nd}}e^{-c \sum_{j\in \AA}|u_{j}|^2(l_{j+1}-l_j)^{2H}}\prod_{j\in \AA}(|u_{j}^{\frac {k}2\alpha_j}|)du,
\end{align*}
where $|J|$ is the Jacobian determinant of changing  variables
$(p_{1}, ~p_{2}, \cdots, p_{n})$ to $(u_{j}, j\in\AA)$.

Therefore, we may reduce the convergence of $\left|\mathbb{E}[(\widehat{\alpha}^{(k)}_{t,\varepsilon}(0))^n]\right|$ to show the finiteness of
$$\int_{E^n}\int_{\mathbb{R}^{nd}}e^{-c \sum_{j\in \AA}|u_{j}|^2(l_{j+1}-l_j)^{2H}}\prod_{j\in \AA}(|u_{j}^{\frac k2\alpha_j}|)\prod_{j\in \AA^c}(l_{j+1}-l_j)^{-\frac{H|k\alpha_j|}{2}}dudl=:\Lambda.$$

Since
$$\int_{\mathbb{R}}e^{-c ~u^2_{ij}(l_{j+1}-l_j)^{2H}}|u_{ij}|^{\frac{k_i}{2}\alpha_{i,j}}du_{ij}\leq C(l_{j+1}-l_j)^{-H-\frac{H}{2}k_i\alpha_{i,j} },$$
we have
\begin{align*}
\Lambda&\leq C \int_{E^n}\prod_{j\in \AA}(l_{j+1}-l_j)^{-Hd-\frac{H}2|k\alpha_j|}\prod_{j\in \AA^c}(l_{j+1}-l_j)^{-\frac{H|k\alpha_j|}{2}}dl\\
&\leq  C_{n,H,t}\int_{E^n}\prod_{j\in \AA}(l_{j+1}-l_j)^{-Hd-H|k|}\prod_{j\in \AA^c}(l_{j+1}-l_j)^{-H|k|}dl\\
&\leq C_{n,H,t}\frac{\Gamma^{n}(1-Hd-H|k|)\Gamma^{n-1}(1-H|k|)}{\Gamma\Big(n(1-Hd-H|k|)+(n-1)(1-H|k|)+1\Big)},
\end{align*}
where $C_{n,H,t}$ is a constant dependent on $n, ~H$ and $t$.

Thus, we can see that $\left|\mathbb{E}[(\widehat{\alpha}^{(k)}_{t,\varepsilon}(0))^n]\right|$ is finite for all $\varepsilon>0$ under condition $H(d+|k|)<1$. Then, we need to prove  $\{\widehat{\alpha}^{(k)}_{t,\varepsilon}(0)\}_{\varepsilon>0}$ is a Cauchy sequence.

Notice that, for any $\varepsilon_1, ~\varepsilon_2>0$,
\begin{align*}
\left|\mathbb{E}\Big[(\widehat{\alpha}^{(k)}_{t,\varepsilon_1}(0)-\widehat{\alpha}^{(k)}_{t,\varepsilon_2}(0))^n\Big]\right|&\leq C \int_{D^n}\int_{\mathbb{R}^{nd}}\prod_{j=1}^n|e^{-\frac{\varepsilon_1}{2}|p_j|^2}-e^{-\frac{\varepsilon_2}{2}|p_j|^2}|\\
&\qquad\qquad\qquad\times\prod_{j=1}^n|p_j^k|\mathbb{E}\Big[\prod_{j=1}^ne^{i\langle p_j, B^H_{s_j}-B^H_{r_j}\rangle}\Big]dp dr ds.
\end{align*}
By the dominated convergence theorem and
\begin{align*}
 \int_{D^n}\int_{\mathbb{R}^{nd}}
\prod_{j=1}^n|p_j^k|\mathbb{E}\Big[\prod_{j=1}^ne^{i\langle p_j, B^H_{s_j}-B^H_{r_j}\rangle}\Big]dp dr ds<\infty,
\end{align*}
we can obtain the desired result. This completes the proof.

\subsection{Proof of Theorem \ref{sec1-th H}}

Let us first prove \eqref{sec1-eq1.3}. For any $\lambda\in[0,1]$, we have the following inequalities
$$|e^{-i\langle p_j, x\rangle}-e^{-i\langle p_j, y\rangle}|\leq C |p_j|^\lambda|x-y|^{\lambda}$$
and
\begin{align*}
|p_j|^\lambda=(p_{1j}^2+\cdots+p_{dj}^2)^{\frac{\lambda}2}\leq C\Big(|p_{1j}|^{\lambda}+\cdots +|p_{dj}|^{\lambda}\Big).
\end{align*}

Using the similar methods as in the proof of Theorem \ref{sec1-th Lp}, we find that
\begin{align*}
\left|\mathbb{E}\Big[\Big(\widehat{\alpha}^{(k)}_{t}(x)-\widehat{\alpha}^{(k)}_{t}(y)\Big)^n\Big]\right|&\leq C |x-y|^{n\lambda} \int_{D^n}\int_{\mathbb{R}^{nd}}\prod_{j=1}^n|p_j^k|\Big(|p_{1j}|^{\lambda}+\cdots +|p_{dj}|^{\lambda}\Big)\\
&\qquad\qquad\qquad\times\mathbb{E}\Big[\prod_{j=1}^ne^{i\langle p_j, B^H_{s_j}-B^H_{r_j}\rangle}\Big]dp dr ds\\
&\leq C |x-y|^{n\lambda} \sum_{l=1}^d\int_{D^n}\int_{\mathbb{R}^{nd}}\prod_{j=1}^n|p_j^{k+\tilde{\lambda}_l}|\mathbb{E}\Big[\prod_{j=1}^ne^{i\langle p_j, B^H_{s_j}-B^H_{r_j}\rangle}\Big]dp dr ds\\
&=:C |x-y|^{n\lambda}\Lambda_1,
\end{align*}
where $\tilde{\lambda}_l=(\lambda_1,\cdots, \lambda_d)$ with $\lambda_l=\lambda$ and all other $\lambda_j=0$. So $|k+\tilde{\lambda}|=|k|+\lambda$
and $\Lambda_1$ is less than (with $\AA$ defined as before)
\begin{align*}
&C \int_{E^n}\prod_{j\in \AA}(l_{j+1}-l_j)^{-Hd-\frac{H}2|\alpha_j(k+\tilde{\lambda})|}\prod_{j\in \AA^c}(l_{j+1}-l_j)^{-\frac{H}2|\alpha_j(k+\tilde{\lambda})|} dl\\
&\qquad\leq C_{n,H,t}\frac{\Gamma^{n}(1-Hd-H|k+\tilde{\lambda}|)\Gamma^{n-1}(1-H|k+\tilde{\lambda}|)}{\Gamma\Big(n(1-Hd-H|k+\tilde{\lambda}|)+(n-1)(1-H|k+\tilde{\lambda}|)+1\Big)},
\end{align*}
which is finite if $1-Hd-H|k+\tilde{\lambda}|>0$.

For the proof of \eqref{sec1-eq1.4}, let $\tilde{D}=\{(r,s): 0<r<s<\tilde{t}\}$ and without loss of generality, we assume that $t<\tilde{t}$. Then
\begin{align*}
\left|\mathbb{E}\Big[\Big(\widehat{\alpha}^{(k)}_{t}(y)-\widehat{\alpha}^{(k)}_{\tilde{t}}(y)\Big)^n\Big]\right|&\leq C \int_{(\tilde{D}\setminus D)^n}\int_{\mathbb{R}^{nd}}\prod_{j=1}^n|p_j^k|\mathbb{E}\Big[\prod_{j=1}^ne^{i\langle p_j, B^H_{s_j}-B^H_{r_j}\rangle}\Big]dp dr ds\\
&\leq C\int_{[t,\tilde{t}]^n}\int_{[0,s_1]\times\cdots\times [0,s_n]}\int_{\mathbb{R}^{nd}}\prod_{j=1}^n|p_j^k|\mathbb{E}\Big[\prod_{j=1}^ne^{i\langle p_j, B^H_{s_j}-B^H_{r_j}\rangle}\Big]dp dr ds\\
&\leq C\int_{\tilde{D}^{n}}\prod_{j=1}^n1_{[t,\tilde{t}]}(s_j)\int_{\mathbb{R}^{nd}}\prod_{j=1}^n|p_j^k|\mathbb{E}\Big[\prod_{j=1}^ne^{i\langle p_j, B^H_{s_j}-B^H_{r_j}\rangle}\Big]dp dr ds\\
&\leq C|t-\tilde{t}|^{n\beta}\Big(\int_{\tilde{D}^{n}}\Big(\int_{\mathbb{R}^{nd}}\prod_{j=1}^n|p_j^k|\mathbb{E}\Big[\prod_{j=1}^ne^{i\langle p_j, B^H_{s_j}-B^H_{r_j}\rangle}\Big]dp\Big)^{\frac1{1-\beta}}dr ds\Big)^{1-\beta}\\
&=:C|t-\tilde{t}|^{n\beta}\Lambda_2,
\end{align*}
where we use the H\"{o}lder's inequality in the last inequality with $\beta<1-H|k|-Hd$.

Using the  similar methods as in the proof of Theorem \ref{sec1-th Lp}, $\Lambda_2$ is bounded by
\begin{align*}
\Big(\int_{E^n}\prod_{j\in \AA}(l_{j+1}-l_j)^{-\frac{Hd}{1-\beta}-\frac{H}{2(1-\beta)}|k\alpha_j|}\prod_{\AA^c}(l_{j+1}-l_j)^{-\frac{H}{2(1-\beta)}|k\alpha_j|} dl\Big)^{1-\beta}.
\end{align*}

Since $1-\beta>H(|k|+d)$, there exists a constant $C>0$, such that
\begin{align*}
\left|\mathbb{E}\Big[\Big(\widehat{\alpha}^{(k)}_{t}(y)-\widehat{\alpha}^{(k)}_{\tilde{t}}(y)\Big)^n\Big]\right|
\leq C|t-\tilde{t}|^{n\beta}.
\end{align*}

\subsection{Proof of Theorem \ref{sec1-th sup}}
In this section, we  mainly use the method  given in Jaramillo and Nualart \cite{Jaramillo 2017}.  We first give the chaos decomposition of the random variable $\widehat{\alpha}^{(k)}_{t,\varepsilon}(0)$ defined in \eqref{sec1-eq1.2} with $d=1$ and $k=1$. We write
$$\widehat{\alpha}^{'}_{t,\varepsilon}(0)=\int_0^t\int_0^s\alpha_{\varepsilon,s,r}drds,$$
where
$$
\alpha_{\varepsilon,s,r}=f'_{\varepsilon}(B_s^H-B_r^H)=\sum_{q=1}^\infty I_{2q-1}(f_{2q-1,\varepsilon,s,r})
$$
with
$$f_{2q-1,\varepsilon,s,r}(x_1,\cdots,x_{2q-1})=(-1)^q\beta_q\Big(\varepsilon+(s-r)^{2H}\Big)^{-q-\frac12}\prod_{j=1}^{2q-1}\mathds{1}_{[r,s]}(x_j)$$
and
$$\beta_q=\frac1{2^{q-\frac12}(q-1)!\sqrt{\pi}}.$$

Then, $\widehat{\alpha}^{'}_{t,\varepsilon}(0)$ has the following chaos decomposition
$$
\widehat{\alpha}^{'}_{t,\varepsilon}(0)=\sum_{q=1}^\infty I_{2q-1}(f_{2q-1,\varepsilon}),
$$
where
$$f_{2q-1,\varepsilon}(x_1,\cdots,x_{2q-1})=\int_{D}f_{2q-1,\varepsilon,s,r}(x_1,\cdots,x_{2q-1})drds$$
with $D=\{(r,s): 0<r<s<t\}$.

For $q=1$,
\begin{equation}\label{sec2-eq2.3}
\mathbb{E}\Big[\Big|I_1(f_{1,\varepsilon})\Big|^2\Big]=\int_{D^2}\langle f_{1,\varepsilon,s_1,r_1},f_{1,\varepsilon,s_2,r_2}\rangle_{\mathfrak{H}}dr_1dr_2ds_1ds_2,
\end{equation}
where $\mathfrak{H}$ is the Hilbert space obtained by taking the completion of the step functions endowed with the inner product
$$\langle \mathds{1}_{[a,b]}, \mathds{1}_{[c,d]}\rangle_{\mathfrak{H}}:=\mathbb{E}[(B_b^H-B_a^H)(B_d^H-B_c^H)].$$

For $q>1$, we have to describe the terms $\langle f_{2q-1,\varepsilon,s_1,r_1},f_{2q-1,\varepsilon,s_2,r_2}\rangle_{\mathfrak{H}^{\otimes(2q-1)}}$,
where $\mathfrak{H}^{\otimes(2q-1)}$ is the $(2q-1)$-th tensor product of $\mathfrak{H}$. For every $x, ~u_1, ~u_2>0$, we define
$$\mu(x,u_1,u_2)=\mathbb{E}[B_{u_1}^H(B_{x+u_2}^H-B_x^H)].$$

Then, from equation (2.19) in  Jaramillo and Nualart \cite{Jaramillo 2017},
$$
\langle f_{2q-1,\varepsilon,s_1,r_1},f_{2q-1,\varepsilon,s_2,r_2}\rangle_{\mathfrak{H}^{\otimes(2q-1)}}=\beta_q^2G^{(q)}_{\varepsilon,r_2-r_1}(s_1-r_1,s_2-r_2),
$$
where
$$G^{(q)}_{\varepsilon,x}(u_1,u_2)=\Big(\varepsilon+u_1^{2H}\Big)^{-\frac12-q}\Big(\varepsilon+u_2^{2H}\Big)^{-\frac12-q}\mu(x,u_1,u_2)^{2q-1}.$$

Before we give the proof of the main result, we give some useful lemmas below. In the sequel, we just consider the case $H=\frac23$.

\begin{lemma} \label{sec3-lem3.4}
$$
\lim_{\varepsilon\to0}\mathbb{E}\Big[\Big|\frac{1}{\log \frac1\varepsilon}\widehat{\alpha}^{'}_{t,\varepsilon}(0)\Big|^2\Big]=\sigma^2.
$$
\end{lemma}

\begin{proof}
From Lemma 5.1 in Jaramillo and Nualart \cite{Jaramillo 2017}, we can see
$$
\mathbb{E}\Big[\Big|\frac{1}{\log \frac1\varepsilon}\widehat{\alpha}^{'}_{t,\varepsilon}(0)\Big|^2\Big]=\frac{1}{(\log \frac1\varepsilon)^2}\Big(V_1(\varepsilon)+V_2(\varepsilon)+V_3(\varepsilon)\Big)
$$
and
\begin{align*}
V_i(\varepsilon)=\frac1{\pi}\int_{D_i}|\varepsilon I+\Sigma|^{-\frac32}\mu drdsdr'ds'
\end{align*}
where $D_i$ (i=1, 2, 3) defined in Lemma \ref{sec3-lem3.2}
and $\Sigma$ is the covariance matrix of $(B^H_s-B^H_r,B^H_{s'}-B^H_{r'})$ with $\Sigma_{1,1}=\lambda$, $\Sigma_{2,2}=\rho$,  $\Sigma_{1,2}=\mu$ given in Lemma \ref{sec3-lem3.2}.

Next, we will split the proof into three parts to consider $V_1(\varepsilon)$, $V_2(\varepsilon)$ and $V_3(\varepsilon)$,  respectively.

\textbf{For the $V_1(\varepsilon)$ term}, changing the coordinates $(r, r', s, s')$ by $(r, a=r'-r, b=s-r', c=s'-s)$ and integrating the $r$ variable, we get
\begin{align*}
V_1(\varepsilon)
&\leq\frac1{\pi}\int_{[0,t]^4}|\varepsilon I+\Sigma|^{-\frac32}\mu drdadbdc\\
&=\frac{t}{\pi}\int_{[0,t]^3}|\varepsilon I+\Sigma|^{-\frac32}\mu dadbdc\\
&=:\widetilde{V_1}(\varepsilon).
\end{align*}

Since
$$\mu=\frac12((a+b+c)^{\frac43}+b^{\frac43}-a^{\frac43}-c^{\frac43})\leq\sqrt{\lambda\rho}=(a+b)^{\frac23}(b+c)^{\frac23}$$
and
\begin{align*}
|\varepsilon I+\Sigma|&=(\varepsilon+\Sigma_{1,1})(\varepsilon+\Sigma_{2,2})-\Sigma^2_{1,2}\\
&\geq C\Big[\varepsilon^2+\varepsilon((a+b)^{\frac43}+(b+c)^{\frac43})+a^{\frac43}(c+b)^{\frac43}+c^{\frac43}(a+b)^{\frac43}\Big]\\
&\geq C\Big[\varepsilon^2+(a+b)^{\frac23}(b+c)^{\frac23}(\varepsilon+(ac)^{\frac23})\Big]\\
&\geq C(a+b)^{\frac23}(b+c)^{\frac23}(\varepsilon+(ac)^{\frac23}),
\end{align*}
where we use the Young's inequality in the second to last inequality.

Then, we have
$$\widetilde{V_1}(\varepsilon)\leq C\int_{[0,t]^3}(a+b)^{-\frac{1}{3}}(b+c)^{-\frac{1}{3}}\Big(\varepsilon+(ac)^{\frac23}\Big)^{-\frac{3}{2}}dadbdc.$$

We will estimate this integral over the regions $\{b\leq (a\vee c)\}$ and $\{b > (a\vee c)\}$ separately, and we will denote these two integrals by $\widetilde{V_{1,1}}$ and $\widetilde{V_{1,2}}$, respectively. If $b\leq (a\vee c)$, without loss of generality, we can assume $a\geq c$ and thus $b\leq a$. For a given small enough constant $\varepsilon_1>0$,

\begin{align*}
\widetilde{V_{1,1}}(\varepsilon)&\leq C\int_{[0,t]^3}(a+b)^{-\frac{1}{3}-\varepsilon_1}(b+c)^{-\frac{1}{3}}\frac{(a+b)^{\varepsilon_1}}{a^{\varepsilon_1}}a^{\varepsilon_1}\Big(\varepsilon+(ac)^{\frac23}\Big)^{-\frac{3}{2}}dadbdc\\
&\leq C\int_{[0,t]^3}b^{-\frac23-\varepsilon_1}a^{\varepsilon_1}\Big(\varepsilon+(ac)^{\frac23}\Big)^{-\frac32}dadbdc\\
&\leq C\int_0^t\int_0^{t\varepsilon^{-\frac32}}a^{\varepsilon_1}\Big(1+(ac)^{\frac23}\Big)^{-\frac32}dcda,
\end{align*}
where we make the change of variable $c=c\,\varepsilon^{-\frac32}$ in the last inequality.

By L'H\^{o}spital's rule, we have
$$
\limsup_{\varepsilon\to0}\frac{\widetilde{V_{1,1}}(\varepsilon)}{\log \frac1\varepsilon}\leq C<\infty.
$$

If $b>(a\vee c)$, we can see that
$$\mu=\frac12((a+b+c)^{\frac43}+b^{\frac43}-a^{\frac43}-c^{\frac43})\leq C\, b^{\frac43}$$
and
\begin{align*}
|\varepsilon I+\Sigma|
&\geq C\Big[\varepsilon^2+\varepsilon((a+b)^{\frac43}+(b+c)^{\frac43})+a^{\frac43}(c+b)^{\frac43}+c^{\frac43}(a+b)^{\frac43}\Big]\\
&\geq C\, b^{\frac43}(\varepsilon+(a\vee c)^{\frac43}).
\end{align*}

Then
\begin{align*}
\limsup_{\varepsilon\to0}\frac{\widetilde{V_{1,2}}(\varepsilon)}{\log \frac1\varepsilon}&\leq \limsup_{\varepsilon\to0}\frac{C}{\log \frac1\varepsilon}\int_{[0,t]^3}b^{-\frac23}\Big(\varepsilon+(a\vee c)^{\frac43}\Big)^{-\frac{3}{2}}dadbdc\\
&\leq \limsup_{\varepsilon\to0}\frac{C}{\log \frac1\varepsilon}\int_0^t\int_0^a(\varepsilon+a^{\frac43})^{-\frac32}dcda\\
&=\limsup_{\varepsilon\to0}\frac{C}{\log \frac1\varepsilon}\int_0^ta(\varepsilon+a^{\frac43})^{-\frac32}da<\infty.
\end{align*}
So, by the above result, we can obtain
\begin{equation}\label{sec3-eq3.7}
\lim_{\varepsilon\to0}\frac{V_{1}(\varepsilon)}{(\log\frac1\varepsilon)^2}=0.
\end{equation}

\textbf{For the $V_2(\varepsilon)$ term}, changing the coordinates $(r, r', s, s')$ by $(r, a=r'-r, b=s'-r', c=s-s')$ and integrating the $r$ variable, we get
\begin{align*}
V_2(\varepsilon)
\leq\frac{t}{\pi}\int_{[0,t]^3}|\varepsilon I+\Sigma|^{-\frac32}\mu dadbdc=:\widetilde{V_2}(\varepsilon).
\end{align*}

By
\begin{align*}
\mu&=\frac12\Big((a+b)^{\frac43}+(b+c)^{\frac43}-a^{\frac43}-c^{\frac43}\Big)\\
&=\frac{2b}3\int_0^1\left((a+bv)^{\frac13}+(c+bv)^{\frac13}\right)dv\\
&\leq \frac43b(a+b+c)^{\frac13}
\end{align*}
and
$$|\varepsilon I+\Sigma|=(\varepsilon+\Sigma_{1,1})(\varepsilon+\Sigma_{2,2})-\Sigma^2_{1,2}\geq \varepsilon^2+\varepsilon((a+b+c)^{\frac43}+b^{\frac43})+C\, b^{\frac43}(a^{\frac43}+c^{\frac43}),$$
we have
$$\widetilde{V_2}(\varepsilon)\leq C\int_{[0,t]^3}b(a+b+c)^{\frac13}\Big(\varepsilon((a+b+c)^{\frac43}+b^{\frac43})+ b^{\frac43}(a^{\frac43}+c^{\frac43})\Big)^{-\frac32}dadbdc.$$

We again estimate this integral over the regions $\{b\leq (a\vee c)\}$ and $\{b > (a\vee c)\}$ separately, and denote these two integrals by $\widetilde{V_{2,1}}$ and $\widetilde{V_{2,2}}$, respectively. If $b\leq (a\vee c)$,
\begin{align*}
\widetilde{V_{2,1}}(\varepsilon)&\leq C\int_{[0,t]^3}b(a\vee c)^{\frac13}\Big(\varepsilon(a\vee c)^{\frac43}+b^{\frac43}(a\vee c)^{\frac43}\Big)^{-\frac{3}{2}}dadbdc\\
&\leq C\int_{[0,t]^3}b(a\vee c)^{-\frac52}\Big(\varepsilon+b^{\frac43}\Big)^{-\frac32}dadbdc\\
&\leq C \int_0^tb\Big(\varepsilon+b^{\frac43}\Big)^{-\frac32}db\\
&\leq C\int_0^{t\varepsilon^{-\frac34}}b(1+b^{\frac43})^{-\frac32}db.
\end{align*}

Thus,
\begin{equation}\label{sec3-eq3.8}
\begin{split}
\limsup_{\varepsilon\to0}\frac{\widetilde{V_{2,1}}(\varepsilon)}{\log \frac1\varepsilon}\leq C\limsup_{\varepsilon\to0}\varepsilon^{-\frac3{2}}(1+t^{\frac43}\varepsilon^{-1})^{-\frac32}<\infty.
\end{split}
\end{equation}

If $b>(a\vee c)$, similarly, we have
\begin{equation}\label{sec3-eq3.9}
\begin{split}
\limsup_{\varepsilon\to0}\frac{\widetilde{V_{2,2}}(\varepsilon)}{\log \frac1\varepsilon}
&\leq \limsup_{\varepsilon\to0}\frac{C}{\log \frac1\varepsilon}\int_0^tb^{-\frac23}db\int_{[0,t]^2}\Big(\varepsilon+(a\vee c)^{\frac43}\Big)^{-\frac32}dadc\\
&\leq \limsup_{\varepsilon\to0}\frac{C}{\log \frac1\varepsilon}\int_0^t\int_0^a(\varepsilon+a^{\frac43})^{-\frac32}dcda\\
&=\limsup_{\varepsilon\to0}\frac{C}{\log \frac1\varepsilon}\int_0^ta(\varepsilon+a^{\frac43})^{-\frac32}da<\infty.
\end{split}
\end{equation}
So, by the above result, we can obtain
$$
\lim_{\varepsilon\to0}\frac{V_{2}(\varepsilon)}{(\log\frac1\varepsilon)^2}=0.
$$

\textbf{For the $V_3(\varepsilon)$ term}.  We first change the coordinates $(r, r', s, s')$ by $(r, a=s-r, b=r'-s, c=s'-r')$ and then by
\begin{align*}
\mu&=\frac12\Big((a+b+c)^{\frac43}+b^{\frac43}-(b+c)^{\frac43}-(a+b)^{\frac43}\Big)\\
&=\frac29ac\int_0^1\int_0^1(b+ax+cy)^{-\frac23}dxdy\\
&=:\mu(a+b,a,c),
\end{align*}
and $|\varepsilon I+\Sigma|=\varepsilon^2+\varepsilon(a^{\frac43}+c^{\frac43})+(ac)^{\frac43}-\mu(a+b,a,c)^2$,
we can find
\begin{align*}
V_3(\varepsilon)
&=\frac{1}{\pi}\int_{[0,t]^3}\mathds{1}_{(0,t)}(a+b+c)(t-a-b-c)\mu|\varepsilon I+\Sigma|^{-\frac32}dadbdc\\
&=\frac1{\pi} \int_{[0,t\varepsilon^{-\frac3{4}}]^2\times[0,t]}\mathds{1}_{(0,t)}(b+\varepsilon^{\frac3{4}}(a+c))\\
&\qquad\qquad\times \frac{(t-b-\varepsilon^{\frac3{4}}(a+c))\varepsilon^{-\frac3{2}}\mu(\varepsilon^{\frac3{4}}a+b,\varepsilon^{\frac3{4}}a,\varepsilon^{\frac3{4}}c)}{\Big[(1+a^{\frac43})(1+c^{\frac43})-\varepsilon^{-2}\mu(\varepsilon^{\frac3{4}}a+b,\varepsilon^{\frac3{4}}a,\varepsilon^{\frac3{4}}c)^2\Big]^{\frac32}}dbdadc,
\end{align*}
where we change the coordinates $(a,b,c)$ by $(\varepsilon^{-\frac3{4}}a, b, \varepsilon^{-\frac3{4}}c)$ in the last equality.

By the definition of $\mu(a+b,a,c)$, it is easy to find
\begin{align*}
\mu(\varepsilon^{\frac3{4}}a+b,\varepsilon^{\frac3{4}}a,\varepsilon^{\frac3{4}}c)&= \frac29
\varepsilon^{\frac3{2}}ac\int_{[0,1]^2}(b+\varepsilon^{\frac3{4}}av_1+\varepsilon^{\frac3{4}}cv_2)^{-\frac23}dv_1dv_2
\end{align*}
and
$$\varepsilon^{-\frac3{2}}\mu(\varepsilon^{\frac3{4}}a+b,\varepsilon^{\frac3{4}}a,\varepsilon^{\frac3{4}}c)=\frac29acb^{-\frac23}+O(\varepsilon^{\frac3{4}}ac(a+c)).$$

The other part of the integrand in $V_3(\varepsilon)$ is
\begin{align*}
&\Big[(1+a^{\frac43})(1+c^{\frac43})-\varepsilon^{-2}\mu(\varepsilon^{\frac3{4}}a+b,\varepsilon^{\frac3{4}}a,\varepsilon^{\frac3{4}}c)^2\Big]^{-\frac32}\\
&\qquad=\Big[(1+a^{\frac43})(1+c^{\frac43})\Big]^{-\frac32}+O\Big(\varepsilon a^2c^2[(1+a^{\frac43})(1+c^{\frac43})]^{-\frac52}\Big).
\end{align*}

Since
\begin{equation}\label{sec3-eq3.11-}
\begin{split}
&\frac{1}{(\log \frac1\varepsilon)^2}\int_{[0,t\varepsilon^{-\frac3{4}}]^2}\varepsilon^{\frac3{4}}ac(a+c)\Big[(1+a^{\frac43})(1+c^{\frac43})\Big]^{-\frac32}dadc\\
&\qquad+\frac{1}{(\log \frac1\varepsilon)^2}\int_{[0,t\varepsilon^{-\frac3{4}}]^2}\varepsilon a^3c^3\Big[(1+a^{\frac43})(1+c^{\frac43})\Big]^{-\frac52}dadc\\
&\to0,
\end{split}
\end{equation}
as $\varepsilon\to0$. Then, by L'H\^{o}spital's rule, we have
\begin{equation}\label{sec3-eq3.11}
\begin{split}
\lim_{\varepsilon\to0}\frac{V_3(\varepsilon)}{(\log \frac1\varepsilon)^2}&=\frac{2}{9\pi}\int_0^t(t-b)b^{-\frac23}db\\
&\qquad\qquad\times\lim_{\varepsilon\to0}\frac{1}{(\log \frac1\varepsilon)^2}\int_{[0,t\varepsilon^{-\frac3{4}}]^2}ac\Big[(1+a^{\frac43})(1+c^{\frac43})\Big]^{-\frac32}dadc\\
&=\frac{t^{\frac43}}{8\pi }B\left(2,\frac13\right).
\end{split}
\end{equation}

Together \eqref{sec3-eq3.7}--\eqref{sec3-eq3.11}, we can see
$$
\lim_{\varepsilon\to0}\mathbb{E}\Big[\Big|\frac{1}{\log \frac1\varepsilon}\widehat{\alpha}^{'}_{t,\varepsilon}(0)\Big|^2\Big]=\frac{t^{\frac43}}{8\pi }B\left(2,\frac13\right)=:\sigma^2.
$$
\end{proof}

\begin{lemma} \label{sec3-lem3.5}
For $I_1(f_{1,\varepsilon})$ given in \eqref{sec2-eq2.3}, then
$$
\lim_{\varepsilon\to0}\mathbb{E}\Big[\Big|\frac{1}{\log \frac1\varepsilon}I_1(f_{1,\varepsilon})\Big|^2\Big]=\sigma^2.
$$
\end{lemma}

\begin{proof}
Form \eqref{sec2-eq2.3}, we can find
$$
\mathbb{E}\Big[\Big|\frac{1}{\log \frac1\varepsilon}I_1(f_{1,\varepsilon})\Big|^2\Big]=\frac{1}{(\log \frac1\varepsilon)^2}\Big(V_1^{(1)}(\varepsilon)+V_2^{(1)}(\varepsilon)+V_3^{(1)}(\varepsilon)\Big),
$$
where $V_i^{(1)}(\varepsilon)=2\int_{D_i}\langle f_{1,\varepsilon,s_1,r_1},f_{1,\varepsilon,s_2,r_2}\rangle_{\mathfrak{H}}dr_1dr_2ds_1ds_2$
for $i=1, 2, 3$. Then we have
\begin{equation} \label{sec3-eq3.15}
0\leq V_i^{(1)}(\varepsilon)\leq V_i(\varepsilon),
\end{equation}
since  $H>\frac12$ and $\mu$ can only take positive values.

Combining  \eqref{sec3-eq3.15} with \eqref{sec3-eq3.7}--\eqref{sec3-eq3.9}, we can see
$$
\lim_{\varepsilon\to0}\frac{1}{(\log \frac1\varepsilon)^2}\Big(V_1^{(1)}(\varepsilon)+V_2^{(1)}(\varepsilon)\Big)=0.
$$
Thus, we only need to consider $\frac{1}{(\log \frac1\varepsilon)^2}V_3^{(1)}(\varepsilon)$ as $\varepsilon\to0$.

By the proof of Lemma 5.7 in Jaramillo and Nualart \cite{Jaramillo 2017}, we have
\begin{align*}
V_3^{(1)}(\varepsilon)&=\frac1{\pi}\int_{S_3}G^{(1)}_{\varepsilon,r'-r}(s-r,s'-r')\\
&=\frac1{\pi}\int_{[0,t]^3}\int_0^{t-(a+b+c)}\mathds{1}_{(0,t)}(a+b+c)(\varepsilon+a^{\frac43})^{-\frac32}(\varepsilon+c^{\frac43})^{-\frac32}\mu(a+b,a,c)ds_1dadbdc\\
&=\frac{2}{9\pi}\int_0^t\int_{[0,t\varepsilon^{-\frac3{4}}]^2}\int_{[0,1]^2}\mathds{1}_{(0,t)}\Big((b+\varepsilon^{\frac3{4}}(a+c)\Big)\Big(t-b-\varepsilon^{\frac3{4}}(a+c)\Big)\\
&\qquad\qquad\qquad \times \Big[(1+a^{\frac43})(1+c^{\frac43})\Big]^{-\frac32}ac\Big(b+\varepsilon^{\frac3{4}}(av_1+cv_2)\Big)^{-\frac23}dv_1dv_2dadcdb.
\end{align*}
Note that
\begin{align*}
\int_{[0,1]^2}\Big(b+\varepsilon^{\frac3{4}}(av_1+cv_2)\Big)^{-\frac23}dv_1dv_2=b^{-\frac23}+O(\varepsilon^{\frac3{4}}(a+c))
\end{align*}
and
\begin{align*}
\int_{[0,1]^2}&\Big(t-b-\varepsilon^{\frac3{4}}(a+c)\Big) \Big[(1+a^{\frac43})(1+c^{\frac43})\Big]^{-\frac32}ac\Big(b+\varepsilon^{\frac3{4}}(av_1+cv_2)\Big)^{-\frac23}dv_1dv_2\\
&=(t-b)b^{-\frac23}ac\Big[(1+a^{\frac43})(1+c^{\frac43})\Big]^{-\frac32}+O\left(\varepsilon^{\frac3{4}}(a+c)ac\Big[(1+a^{\frac43})(1+c^{\frac43})\Big]^{-\frac32}\right).
\end{align*}

Similar to \eqref{sec3-eq3.11-} and \eqref{sec3-eq3.11}, we can find that
$$\lim_{\varepsilon\to0}\frac1{(\log \frac1\varepsilon)^2}\int_{[0,t\varepsilon^{-\frac3{4}}]^2}\varepsilon^{\frac3{4}}(a+c)ac\Big[(1+a^{\frac43})(1+c^{\frac43})\Big]^{-\frac32}dadc=0$$
and
$$\lim_{\varepsilon\to0}\frac1{(\log \frac1\varepsilon)^2}\int_{[0,t\varepsilon^{-\frac3{4}}]^2}ac\Big[(1+a^{\frac43})(1+c^{\frac43})\Big]^{-\frac32}dadc=\frac9{16}.$$

Thus,
\begin{align*}
\lim_{\varepsilon\to0}\frac{V^{(1)}_3(\varepsilon)}{(\log \frac1\varepsilon)^2}&=\frac{2}{9\pi}\int_0^t\frac9{16}(t-b)b^{-\frac23}db\\
&=\lim_{\varepsilon\to0}\frac{V_3(\varepsilon)}{(\log \frac1\varepsilon)^2}\\
&=\sigma^2.
\end{align*}
\end{proof}

\textbf{Proof of  Theorem  \ref{sec1-th sup}}

By Lemmas \ref{sec3-lem3.4}--\ref{sec3-lem3.5} and
$$\widehat{\alpha}^{'}_{t,\varepsilon}(0)=I_1(f_{1,\varepsilon})+\sum_{q=2}^\infty I_{2q-1}(f_{2q-1,\varepsilon}),$$
we can see
$$\lim_{\varepsilon\to0}\mathbb{E}\Big[\Big|\frac1{\log \frac1{\varepsilon}}\sum_{q=2}^\infty I_{2q-1}(f_{2q-1,\varepsilon})\Big|^2\Big]=0.$$

Since $I_1(f_{1,\varepsilon})$ is Gaussian, then we have,  as $\varepsilon\to0$,
$$
\Big(\log\frac1{\varepsilon}\Big)^{-1}I_1(f_{1,\varepsilon})\overset{law}{\to}N(0,\sigma^2).
$$
Thus,
$$
\Big(\log\frac1{\varepsilon}\Big)^{-1}\widehat{\alpha}^{'}_{t,\varepsilon}(0)\overset{law}{\to}N(0,\sigma^2)
$$
as $\varepsilon\to0$. This completes the proof.
\bigskip

\textbf{Acknowledgement}\\
I would like to sincerely thank my supervisor Professor Fangjun Xu, who has led the way to this work. I'm also grateful to the anonymous referees and the editor for their insightful and valuable comments which have greatly improved the presentation of the paper.
Q. Yu is partially supported by National Natural Science Foundation of China (Grant  No.11871219), ECNU Academic Innovation Promotion Program for Excellent Doctoral Students (Grant No.YBNLTS2019-010) and the Scientific Research Innovation Program for Doctoral Students in the Faculty of Economics and Management (Grant  No. 2018FEM-BCKYB014).

\end{document}